\documentclass[pdflatex,sn-mathphys-num]{sn-jnl}

\usepackage{graphicx}%
\usepackage{multirow}%
\usepackage{amsmath,amssymb,amsfonts}%
\usepackage{amsthm}%
\usepackage{mathrsfs}%
\usepackage[title]{appendix}%
\usepackage{xcolor}%
\usepackage{textcomp}%
\usepackage{manyfoot}%
\usepackage{booktabs}%
\usepackage{algorithm}%
\usepackage{algorithmicx}%
\usepackage{algpseudocode}%
\usepackage{listings}%
\usepackage{lmodern}
\usepackage{geometry}
\usepackage{pgfplots}
\pgfplotsset{compat=1.17}
\usepackage{pgfplots}
\usepackage{pgfplotstable}
\pgfplotsset{compat=newest}
\newtheorem{lemma}{Lemma}
\usepackage{enumitem}

\theoremstyle{plain}%
\newtheorem{theorem}{Theorem}
\newtheorem{proposition}[theorem]{Proposition}%

\theoremstyle{plain}%

\theoremstyle{plain}%
\newtheorem{definition}{Definition}%

\begin{document}

\title[Article Title]{\textbf{Mean Value \text{-}Mann Iterative Process}}


\author*[1]{\fnm{Mohd} \sur{Tariq}}\email{mohd.tariq@myyahoo.com}

\author[2]{\fnm{Mayank} \sur{Sharma}}\email{mohd.tariq@myyahoo.com, mayank.sharma@vitbhopal.ac.in}

\affil[]{\orgdiv{Division of Mathematics School of Advanced Sciences and Languages}, \orgname{VIT Bhopal
University}, \orgaddress{\street{Kothri Kalan}, \city{Sehore}, \postcode{466114}, \state{Madhya Pradesh}, \country{India}}}


\abstract{In this paper, the Mean value iterative process is modified with the Mann iterative process for mean nonexpansive mapping in a hyperbolic metric space that satisfy the symmetry criteria and in uniformly
convex hyperbolic spaces to validate the iterative process, we present strong and \(\Delta\)-convergence theorems.}




\maketitle

\section{Introduction}\label{sec1}
Let $(\mathcal{S}, d)$ be a metric space, defined a mapping $\daleth : \mathcal{A} \to \mathcal{A}$ on a $\mathcal{A \neq \emptyset}$ closed convex subset of $\mathcal{S}$ is said to be 

\begin{enumerate}[label=\textbf{\roman*.}]
     \item \text{Nonexpansive} if \[ d(\daleth\varkappa_1, \daleth\varkappa_2) \leq d(\varkappa_1, \varkappa_2) \quad \textit{for all} \quad \varkappa_1, \varkappa_2\in \mathcal{A} \]
      \item \text{Mean Nonexpansive Mapping}, for \(a,b \geq0\) and \(a+b \leq1\) if \[ d(\daleth\varkappa_1, \daleth\varkappa_2) \leq a \, d(\varkappa_1, \varkappa_2) + b \, d(\varkappa_1, \daleth\varkappa_2) \quad \textit{for all} \quad \varkappa_1, \varkappa_2 \in \mathcal{A}\]
\end{enumerate}

\noindent The idea of mean nonexpansive mappings was presented by Zhang \cite{zhang1975fixed}, who showed that such mappings has a fixed point within a weakly compact convex subset \(\mathcal{A}\) of a Banach space, provided the subset has a normal structure. The study of mean nonexpansive mappings in Banach spaces has since been undertaken by a number of scholars. Some characteristics of mean nonexpansive mappings were examined by Wu and Zhang \cite{wu2007fixed}.  They proved that the mean nonexpansive mapping has a fixed point if \(a + b < 1\). For mean nonexpansive mappings, Zhao \cite{hanbin22existence} showed that the Picard iterative process \cite{picard1890memoire}, \( \varkappa_{n+1}=\daleth\varkappa_n\) and Mann iterative process \cite{mann1953mean}, \(\varkappa_{n+1}=(1-\alpha_n) \varkappa_n+\alpha_n \daleth\varkappa_n\), where \(\{\alpha_n\} \in [0,1]\) converge.  In the context of uniformly convex Banach spaces. In 2007, Gu and Li \cite{gu2007approximating} demonstrated the convergence of the Ishikawa iterative process \cite{ishikawa1974fixed},  \(\varkappa_{n+1}=(1-\alpha_n) \varkappa_n+\alpha_n T((1-\beta_n)\varkappa_n+\beta_n \daleth(\varkappa_n))\), where \(\{\alpha_n\},\{\beta_n\} \in [0,1]\) for mean nonexpansive mappings. 

\vspace{2.4mm}

In 2012, Fixed point results for mean nonexpansive mappings in Hilbert spaces were examined by Ouahab et al. \cite{ouahab2012fixed} and they also established the Chebyshev radius and the modulus of convexity. Zuo \cite{zuo2014fixed} showed in 2014 that mean nonexpansive mappings have approximation fixed point sequences and given the right circumstances gave existence and uniqueness theorems for fixed points. Zhou and Cui \cite{zhou2015fixed} used the Ishikawa iterative process to introduce and study mean nonexpansive mappings in \text{CAT}(0) spaces in 2015. They demonstrated $\Delta$-convergence and strong theorems for the sequence $\{\varkappa_n\}$ produced by the previous iteration. In 2017 Chen et al. \cite{chen2017fixed} first proposed the idea of mean nonexpansive set-valued mappings in Banach spaces.  They expanded Lim's and Nadler's fixed point theorems to include mean nonexpansive set-valued mappings. In 2018, Akbar et al. \cite{abkar2018approximation} developed a new iterative process to approximate the fixed point of mean nonexpansive mappings in \text{CAT}(0) spaces and established convergence results.  In 2021, Ahmad et al. \cite{ahmad2021approximation} proved weak and strong convergence theorems for mean non-expansive mappings in Banach spaces using the Picard-Mann hybrid iterative process. In 2022, Izuchukwu et al. \cite{izuchukwu2021approximating} addressed fixed point qualities and the demiclosedness principle for mean non-expansive mappings in uniformly convex hyperbolic spaces.  Abbas and Nazir \cite{abbas2014new} devised an iterative approach for approximating a common fixed point of two mean nonexpansive mappings, resulting in strong and $\Delta$-Convergence theorems.

\vspace{2.4mm}

To approximate the fixed points of nonlinear mappings in Hilbert and Banach spaces, a number of fixed point results and iterative processs have been created because of their convex structures. Takahashi \cite{takahashi1970convexity} originally established the idea of convex metric spaces by examining the fixed points for nonexpansive mappings.  Different convex structures on metric spaces have since been introduced in a number of attempts.  The hyperbolic space is a metric space that has a convex structure.  As a result of the introduction of various convex structures on hyperbolic spaces, hyperbolic spaces have been defined differently. Kohlenbach \cite{kohlenbach2005some} established a class of hyperbolic spaces that is more structured and better suitable for some applications than the class of hyperbolic spaces. The hyperbolic spaces presented in \cite{kohlenbach2005some, goebel1983iteration, reich1990nonexpansive} include Banach spaces and CAT(0) spaces for mean nonexpansive mappings, results on fixed point problems for mean nonexpansive mappings from uniformly convex Banach spaces and CAT(0) spaces must therefore be extended to uniformly convex hyperbolic spaces, since the class of uniformly convex hyperbolic spaces generalize uniformly convex Banach spaces and CAT(0) spaces.

\section{Preliminaries}
Some definitions and results as part of the structure of hyperbolic spaces. 

\vspace{2.4mm}

\begin{definition}
    \cite{kohlenbach2005some} A metric space $(\mathcal{S},d)$ with a convex mapping $\mathcal{C} : \mathcal{S}^2 \times [0,1] \to \mathcal{S}$ $\forall \varkappa_1,\varkappa_2,\varkappa_3,\varkappa_4 \in \mathcal{S}$ and $a, b \in [0,1]$ satisfying
\begin{itemize}
    \item[(i)] $d(v,\mathcal{C}(\varkappa_1,\varkappa_2,a)) \leq (1-a)d(v,\varkappa_1) + a d(v,\varkappa_2),$
    \item[(ii)] $d(\mathcal{C}(\varkappa_1,\varkappa_2,a), \mathcal{C}(\varkappa_1,\varkappa_2,b)) = |a - b|d(\varkappa_1,\varkappa_2),$
    \item[(iii)] $\mathcal{C}(\varkappa_1,\varkappa_2,a) = \mathcal{C}(\varkappa_2,\varkappa_1,1-a),$
    \item[(iv)] $d(\mathcal{C}(\varkappa_1,\varkappa_3,a), \mathcal{C}(\varkappa_2,\varkappa_4,a)) \leq (1-a)d(\varkappa_1,\varkappa_2) + a d(\varkappa_3,\varkappa_4)$
\end{itemize}
is called a hyperbolic space $(\mathcal{S},d,\mathcal{C}).$
\end{definition}

\vspace{2.4mm}

\begin{definition}
\cite{suanoom2016remark} Assume $\mathcal{S}$ is a real Banach space with the norm $\|.\|$ and define $d : \mathcal{S}^2 \to [0,\infty)$ by \(d(\varkappa_1,\varkappa_2) = \|\varkappa_1 - \varkappa_2\|.\) Then $(\mathcal{S},d,\mathcal{C})$ is a hyperbolic space with $\mathcal{C} : \mathcal{S}^2 \times [0,1] \to \mathcal{S}$ defined by $\mathcal{C}(\varkappa_1,\varkappa_2,a) = (1-a)\varkappa_1 + a \varkappa_2$.
\end{definition}

\vspace{2.4mm}

\begin{definition}
    \cite{suanoom2016remark} Assume $\mathcal{S}$ is a hyperbolic space with $\mathcal{C} : \mathcal{S}^2 \times [0,1] \to \mathcal{S}$ 
\begin{itemize}
    \item[(i)] $\mathcal{A} (\neq \emptyset) \subseteq \mathcal{S}$, for all $\varkappa_1,\varkappa_2 \in \mathcal{A}$ and $a \in [0,1]$ is called a convex if $ \mathcal{C}(\varkappa_1,\varkappa_2,a) \in \mathcal{A}$ .
    \item[(ii)] If $\epsilon \in (0,2]$, and for any $r > 0 \exists$ a $\delta \in (0,1]$  $\forall \varkappa_1,\varkappa_2,\varkappa_3 \in \mathcal{S}$ such that 
    \(d(\mathcal{C}(\varkappa_1,\varkappa_2,\dfrac{1}{2}),\varkappa_3) \leq (1-\delta)r\) if given $d(\varkappa_1,\varkappa_3) \leq r, d(\varkappa_2,\varkappa_3)) \leq r,$ and $d(\varkappa_1,\varkappa_2) \geq \epsilon r$ $\mathcal{S}$ then is said to be uniformly convex.
    \item[(iii)] If $\epsilon \in (0,2]$ and for every $r > 0$, a modulus of uniform convexity of $\mathcal{S}$ is defined $\eta : (0,\infty) \times (0,2] \to (0,1]$ by $\delta = \eta(r,\epsilon)$. $\eta$ is monotone if decreases with $r$ (for a given $\epsilon$). 
\end{itemize}
\end{definition}

\vspace{2.4mm}  

\begin{definition}
     Suppose $\mathcal{A}\neq \emptyset$ subset of a metric space $\mathcal{S}$ and any bounded sequence $\{\varkappa_n\} \in \mathcal{A}$ for $\varkappa \in \mathcal{S}$ defined a continuous functional $r(.,\{\varkappa_n\}): \mathcal{S} \to [0,\infty)$ by \( r(\varkappa,\{\varkappa_n\}) = \limsup_{n \to \infty} d(\varkappa_n,\varkappa).\)
\begin{itemize}
    \item[(i)] The asymptotic radius \(r(\mathcal{A},\{\varkappa_n\}) = \inf \{r(\varkappa,\{\varkappa_n\}): \varkappa \in \mathcal{A}\}\) of $\{\varkappa_n\}$ w.r.t. $\mathcal{A}.$ 
    \item[(ii)] If \(r(\varkappa,\{\varkappa_n\}) = \inf \{r(y,\{\varkappa_n\}): y \in \mathcal{A}\}\) then $\varkappa \in \mathcal{A}$ is an asymptotic center of $\{\varkappa_n\}$ w.r.t. $\mathcal{A} \subseteq \mathcal{S}.$ 
    \item[(iii)] The set $\mathbf{C}(\mathcal{A},\{\varkappa_n\})$ consists of all asymptotic centers of $\{\varkappa_n\}$ w.r.t. $\mathcal{A}.$
    \item[(iv)] The asymptotic radius $r(\{\varkappa_n\})$ and the asymptotic center $\mathbf{C}(\{\varkappa_n\})$ w.r.t $\mathcal{S}$.
\end{itemize}
\end{definition}   

\vspace{2.4mm}  

\begin{definition}
    \cite{kirk2008concept} For any subsequence $\{\varkappa_{n_k}\}$ of $\{\varkappa_n\}$, if there is a unique asymptotic center $\varkappa$  of $\{\varkappa_{n_k}\}$ then $\{\varkappa_n\} \in \mathcal{S}$ is said to be $\Delta$-converge to $\varkappa \in \mathcal{S}$,   denoted by $\Delta$-$\lim\limits_{n \to \infty} \varkappa_n = \varkappa$.
\end{definition}

\vspace{2.4mm}   

\noindent \textbf{Remark 1.} \cite{kuczumow1978almost} In the the Banach space, weak convergence and $\Delta$-convergecoincideide with the usual Opial property.

\vspace{2.4mm}  

\begin{lemma}
\cite{leustean2008nonexpansive} Suppose $\mathcal{A} \neq \emptyset$ closed convex subset of $\mathcal{S}$ and $\mathcal{S}$ with monotone modulus of uniform convexity $\eta$ is a complete uniformly convex hyperbolic space, then any bounded sequence $\{\varkappa_n\} \in \mathcal{S}$ has a unique asymptotic center w.r.t. any \(\mathcal{A}\).
\end{lemma}

\vspace{2.4mm}  

\begin{lemma}
    \cite{chang2014delta} Suppose $\mathcal{S}$ is a complete uniformly convex hyperbolic space with monotone modulus of uniform convexity $\eta$ and $\{\varkappa_n\} \in \mathcal{S}$ is a bounded sequence with $\mathbf{C}(\{\varkappa_n\}) = \{\varkappa\}$, take any subsequence $\{x_{n_k}\}$ of $\{\varkappa_n\}$ with $\mathbf{C}(\{x_{n_k}\}) = \{\varkappa_1\}$ and $\{d(\varkappa_n,\varkappa_1)\}$ converges, then $\varkappa = \varkappa_1$.
\end{lemma}

\vspace{2.4mm}

\begin{lemma}
\cite{khan2012implicit} Assume that $(\mathcal{S},d,\mathcal{C})$ with monotone modulus of uniform convexity $\eta$, $z^* \in \mathcal{S}$ and $\{\varkappa_n\} \in [a,b]$ for some $a,b \in (0,1)$ is a complete uniformly convex hyperbolic space if $\{u_n\},\{v_n\} \in \mathcal{S}$ such that
\[
\limsup_{n \to \infty} d(u_n,z^*) \leq c, \quad \limsup_{n \to \infty} d(v_n,z^*) \leq c, \quad \text{and} \quad \lim\limits_{n \to \infty} d(\mathcal{C}(u_n,v_n,\varkappa_n),z^*) = c,
\]
for some $c > 0$, then $\lim\limits_{n \to \infty} d(u_n,v_n) = 0$.
\end{lemma}

\vspace{2.4mm}

\begin{definition} 
Assume $\mathcal{S}$ is a hyperbolic space, $\mathcal{A (\neq \emptyset)} \subseteq \mathcal{S} $ then $\{\varkappa_n\} \in \mathcal{A}$ is called a Fejér monotone  w.r.t. $\mathcal{A}$ if $\forall \varkappa \in \mathcal{A}$
\[
d(\varkappa_{n+1},\varkappa) \leq d(\varkappa_n,\varkappa), \quad n \in \mathbb{N}
\]
\end{definition}

\vspace{2.4mm}

\begin{proposition}
    \cite{imdad2016fixed} Assume the sequence $\{\varkappa_n\} \in X$ and $\mathcal{A (\neq \emptyset) \subseteq} \mathcal{S}$ and $\daleth: \mathcal{A}\to \mathcal{A}$ is any nonlinear mapping with $\{\varkappa_n\}$ is Fejér monotone w.r.t. $\mathcal{A}$, then
\begin{itemize}
    \item[(i)] $\{\varkappa_n\}$ is bounded.
    \item[(ii)] $\forall z^* \in \mathcal{F}(\daleth)$ $\{d(\varkappa_n,z^*)\}$ is decreasing and converges.
    \item[(iii)] $\lim\limits_{n \to \infty} d(\varkappa_n,\mathcal{F}(\daleth))$ exists.
\end{itemize}

\end{proposition}

\vspace{2.4mm}

\begin{theorem}
    \cite{izuchukwu2021approximating} Suppose $\mathcal{S}$ is a hyperbolic space $\mathcal{S}$, $\mathcal{A}  \neq \emptyset $ is a closed convex subset of $\mathcal{S}$  with $b < 1$ and $\mathcal{F}(\daleth) \neq \emptyset$, if a mean nonexpansive mapping  $\daleth: \mathcal{S} \to \mathcal{S}$ is exist then $\mathcal{F}(\daleth)$ is closed convex.
\end{theorem}

\vspace{2.4mm}

\begin{theorem}
    \cite{izuchukwu2021approximating} Assume $\mathcal{S}$ with monotone modulus of convexity $\eta$ is a complete uniformly convex hyperbolic space, $\mathcal{A}  \neq \emptyset $ closed convex subset of $\mathcal{S}$ with $b < 1$ and $\{\varkappa_n \in \mathcal{A}\}$ is a bounded sequence, defined a mean nonexpansive mapping  $\daleth : \mathcal{A} \to \mathcal{A}$ such that $\lim\limits_{n \to \infty} d(\varkappa_n, T \varkappa_n) = 0$ and $\Delta\text{-}\lim\limits_{n \to \infty} \varkappa_n = z^*$ then $z^* \in \mathcal{F}(\daleth)$.
\end{theorem}

\vspace{2.4mm}

\begin{definition}
    A mapping $\daleth : \mathcal{A} \to \mathcal{A}$ is said to be strictly pseudocontractive if there is a number \( k \), \( 0 \leq k < 1 \), such that if each of \( x \) and \( y \) in \( \mathcal{A} \), then
   \[ \|\daleth\varkappa_1 - \daleth\varkappa_2\|^2 \leq \|\varkappa_1 - \varkappa_2\|^2 + k \| (I - T)\varkappa_1 - (I - T)\varkappa_2 \|^2 \]
\end{definition}

\noindent In 1972, Gordon \cite{johnson1972fixed} presented \textit{the mean value iterative process}    to get to the fixed point through \textit{strictly pseudocontractive mappings} in Hilbert spaces, for any positive integer \( n \) is defined by 
\begin{equation}
    \varkappa_{n+1} = \mathcal{C}(\varkappa_n, \daleth\varkappa_n, n) =[1/(n + 1 )][(n\varkappa_n + \daleth\varkappa_n)]
\end{equation} 

\noindent We present an iterative process that extended to \textit{the hyperbolic space} with a sequence \( \{r_n\} \) for \textit{mean nonexpansive mappings}. In order to preserve convergence properties in this non-Euclidean space, the iterative process must be reinterpreted to account for the unique properties of hyperbolic distances and transformations.

\section{Main Results}\label{sec3}
An iterative process for \(x_0\) in \(\mathcal{A}\), the sequence \( \{\varkappa_n\}_n^\infty \)  in \(\mathcal{A}\) is given by
 
\begin{equation}
 \varkappa_{n+1} = \mathcal{C}(\varkappa_n, \daleth\varkappa_n, \alpha_n, r_n) \\
  =(1-\alpha_n) \varkappa_n+\alpha_n \daleth\left(\dfrac{r_n \varkappa_n}{r_n + 1}+\dfrac{T \varkappa_n}{r_n + 1}\right)
\end{equation}

\text{where} 
\( \alpha_n \in [0,1], r_n \in [0,\infty)\) and equation (2) reduces to 

\vspace{2.4mm}

\begin{itemize}
  \item The two step Picard iteration \(\varkappa_{n+1} = \mathcal{C}(\varkappa_n, \daleth\varkappa_n)\), if \( r_n \to 0 \) and \(\alpha_n=1\).

   \item The Mann iteration \(\varkappa_{n+1} = \mathcal{C}(\varkappa_n, \daleth\varkappa_n,\alpha_n)\), if \( r_n (= 2^n) \to \infty\), as \( n \to \infty \)
    \[
    \varkappa_{n+1} = \mathcal{C}(\varkappa_n, \daleth\varkappa_n, 2^n) = \dfrac{2^n \varkappa_n}{2^n + 1} + \dfrac{T \varkappa_n}{2^n + 1}\approx \dfrac{2^n \varkappa_n}{2^n} = \varkappa_n.
    \]

    \item The  Ishikawa iteration \(\varkappa_{n+1} = \mathcal{C}(\varkappa_n, \daleth\varkappa_n,\alpha_n,\beta_n)\), if \( (r_n + 1)^{-1} =\beta_n \).
\end{itemize}

\vspace{2.4mm}

\begin{lemma}
     Assume $\mathcal{A}  \neq \emptyset $ is a closed convex subset of a hyperbolic space $\mathcal{S}$ and $\daleth : \mathcal{A} \to \mathcal{A}$ is a mean nonexpansive mapping with $\mathcal{F}(\daleth)\neq \emptyset$ and the sequence $\{\varkappa_n\}$ is defined by (2) then $\lim\limits_{n \to \infty} d(\varkappa_n, z^*)$ exists $ \forall z^* \in \mathcal{F}(\daleth)$.
\end{lemma}

\begin{proof}
     If $z^* \in \mathcal{F}(\daleth)$ and from (2) we get
\begin{align*}
d(\varkappa_{n+1}, z^*) &= d(\mathcal{C}(\varkappa_n, \daleth\varkappa_n,\alpha_n, r_n), z^*) \\
&\leq (1-\alpha_n)d(\varkappa_n,z^*) + \alpha_nd\left(\daleth \left(\dfrac{r_n x_n}{r_n+1} + \dfrac{\daleth x_n}{r_n+1}\right),z^*\right)\\
&\leq (1-\alpha_n)d(\varkappa_n,z^*) + \alpha_n(a+b)d \left(\dfrac{r_n x_n}{r_n+1} + \dfrac{\daleth x_n}{r_n+1},z^*\right) \\
&\leq (1-\alpha_n)d(\varkappa_n,z^*) + \alpha_n\left(\dfrac{r_n}{r_n+1}d(\varkappa_n,z^*)+ \dfrac{1}{r_n+1}d(\daleth\varkappa_n,z^*)\right)\\
&\leq (1-\alpha_n)d(\varkappa_n,z^*) + \alpha_n\left(\dfrac{r_n}{r_n+1}d(\varkappa_n,z^*)+ \dfrac{(a+b)}{r_n+1}d(\varkappa_n,z^*)\right) \\
&\leq (1-\alpha_n)d(\varkappa_n,z^*) + \alpha_n d(\varkappa_n,z^*) \\
&\leq d(\varkappa_n,z^*) 
\end{align*}
$ \implies d(\varkappa_n, z^*)$ is non-increasing and bounded. Hence $\lim\limits_{n \to \infty} d(\varkappa_n, z^*)$ exists.
\end{proof}

\vspace{2.4mm}

\begin{lemma}
Suppose \(\mathcal{S}\) is a complete uniformly convex hyperbolic space with monotone modulus of uniform convexity $\eta$ and defined a mean nonexpansive mapping $\daleth : \mathcal{A} \to \mathcal{A}$, $\mathcal{A}  \neq \emptyset $ is a closed convex subset of \(\mathcal{S}\) with $\mathcal{F}(\daleth) \neq \emptyset$, then the sequence $\{\varkappa_n\}$ is defined by (2) is bounded and $\lim\limits_{n \to \infty} d(\varkappa_n, \daleth\varkappa_n) = 0$.
\end{lemma}

\begin{proof}
From Lemma 4, we know that $\lim\limits_{n \to \infty} d(\varkappa_n, z^*)$ exists for each $z^* \in \mathcal{F}(\daleth)$.
\end{proof}

\[  \text{Let} \lim\limits_{n \to \infty} d(\varkappa_n, z^*) = c \implies \limsup_{n \to \infty} d(\varkappa_n, z^*) \leq c\]

\noindent \textbf{Case I:} If $c = 0$, we're done.

\noindent \textbf{Case II:} If $c > 0$, 
\begin{align*}
d(\daleth\varkappa_n, z^*) &\leq a \, d(\varkappa_n, z^*) + b \, d(\varkappa_n, z^*) \\
             &= (a + b) \, d(\varkappa_n, z^*) \\
             &\leq d(\varkappa_n, z^*).
\end{align*}

\[\implies
\limsup_{n \to \infty} d(\daleth\varkappa_n, z^*) \leq c
\]
Let \(\omega_n = \dfrac{r_n x_n}{r_n+1} + \dfrac{\daleth x_n}{r_n+1}\)

\begin{align*}
d(\daleth\omega_n, z^*) &\leq a \, d(\omega_n, z^*) + b \, d(\omega_n, z^*) \\
             &= (a + b) \, d(\omega_n, z^*) \\
             &\leq d(\omega_n, z^*) \\
             &= d\left(\dfrac{r_n x_n}{r_n+1} + \dfrac{\daleth x_n}{r_n+1},z^*\right) \\ &\leq\dfrac{r_n}{r_n+1}d(\varkappa_n,z^*)+ \dfrac{1}{r_n+1}d(\daleth\varkappa_n,z^*) \\
             &\leq d(\varkappa, z^*)
\end{align*}

\[ \implies 
\limsup_{n \to \infty} d(\daleth\omega_n, z^*) \leq c \quad \text{and \quad} \limsup_{n \to \infty} d(\omega_n, z^*) \leq c
\]

Now,

\begin{align*}
c &=  \lim\limits_{n \to \infty} d(\varkappa_{n+1}, z^*)\\
&\leq  \lim\limits_{n \to \infty} \left((1 - \alpha_n) d(\varkappa_n, z^*) + \alpha_n d(\daleth \omega_n, z^*)\right)
\end{align*}
From Lemma 3, we have
\[
\lim\limits_{n \to \infty} d(\varkappa_n, \daleth \omega_n) = 0.
\]

\begin{align*}
d(\varkappa_{n+1}, z^*)
&\leq (1 - \alpha_n) d(\varkappa_n, z^*) + \alpha_n d(\daleth \omega_n, z^*) \\
& \leq (1 - \alpha_n) [d(\varkappa_n, \daleth\omega_n) + d(\daleth\omega_n, z^*)] + \alpha_n d(\daleth \omega_n, z^*) \\
& \leq d(\daleth \omega_n, z^*)
\end{align*}

\[\implies c \leq \liminf_{n \to \infty} d(\daleth\omega_n, z^*)\]

we get

\[\lim_{n \to \infty} d(\daleth\omega_n, z^*) = c\]

and

\begin{align*}
d(\daleth\omega_n, z^*) &\leq a \, d(\omega_n, z^*) + b \, d(\omega_n, z^*) \\
             &= (a + b) \, d(\omega_n, z^*) \\
             &\leq d(\omega_n, z^*).
\end{align*}

\[\implies c \leq \liminf_{n \to \infty} d(\omega_n, z^*)\]

we get

\[\lim_{n \to \infty} d(\omega_n, z^*) = c\]

Now 

\begin{align*}
    c &= \lim_{n \to \infty} d(\omega_n, z^*) \\ 
    &= \lim_{n \to \infty} d\left(\dfrac{r_n x_n}{r_n+1} + \dfrac{\daleth x_n}{r_n+1}, z^*\right) \\ 
    &\leq \lim_{n \to \infty} \left(\dfrac{r_n}{r_n+1}d(x_n,z^*) + \dfrac{1}{r_n+1}d(\daleth x_n,z^*) \right)
\end{align*}

From Lemma 3, we have

\[
\lim\limits_{n \to \infty} d(\varkappa_n, \daleth \varkappa_n) = 0.
\]

\vspace{2.4mm}

\begin{theorem}
Suppose \(\mathcal{S}\) is a complete uniformly convex hyperbolic space with monotone modulus of uniform convexity $\eta$ and defined a mean nonexpansive mapping $\daleth : \mathcal{A} \to \mathcal{A}$, $\mathcal{A}  \neq \emptyset $ is a closed convex subset of \(\mathcal{S}\) for \( b < 1 \) with \( \mathcal{F}(\daleth) \neq \emptyset \) then \( \{\varkappa_n\} \) defined by (2) is \(\Delta\)-convergence to a fixed point of \( \daleth \).
\end{theorem}

\begin{proof}
Since $\{\varkappa_n\}$ is a bounded then $\{\varkappa_n\}$ has $\Delta$-convergent subsequences. Suppose $\omega_1$ and $\omega_2$ are the $\Delta$-limits of subsequences $\{\omega_{n_j\}}$ and $\{\omega_{n_k}\}$ of $\{\varkappa_n\}$ respectively from Lemma (1 \& 5)

\[  \mathbf{C}(\mathcal{A},\{\omega_{n_j}\}) = \{\omega_1\}, \quad  \mathbf{C}(\mathcal{A},\{\omega_{n_k}\}) = \{\omega_2\} \quad \& \lim\limits_{n \to \infty} d(\omega_{n_j}, \daleth\omega_{n_j}) = 0, \quad \lim\limits_{n \to \infty} d(\omega_{n_k}, \daleth\omega_{n_k}) = 0.
\]
using theorem 3, $\omega_1, \omega_2 \in \mathcal{F}(\daleth)$.

\noindent Now, suppose $\omega_1 \neq \omega_2$, since asymptotic center is uniqueness.
\begin{align*}
\limsup_{n \to \infty} d(\varkappa_n, \omega_1) 
&= \limsup_{n \to \infty} d(\omega_{n_j}, \omega_1) \\
&\leq \limsup_{n \to \infty} d(\omega_{n_j} \omega_2) \\
&= \limsup_{n \to \infty} d(\varkappa_n, \omega_2) \\
&= \limsup_{n \to \infty} d(\omega_{n_k}, \omega_2) \\
&= \limsup_{n \to \infty} d(\omega_{n_k}, \omega_1) \\
&= \limsup_{n \to \infty} d(\varkappa_n, \omega_1).
\end{align*}
which is a contradiction, thus $\omega_1 = \omega_2$ and  $\{\varkappa_n\}$ $\Delta$-converges to a fixed point of $\daleth$.
\end{proof}

\begin{theorem}
Suppose \(\mathcal{S}\) is a complete uniformly convex hyperbolic space with monotone modulus of uniform convexity $\eta$ and defined a mean nonexpansive mapping $\daleth : \mathcal{A} \to \mathcal{A}$, $\mathcal{A}  \neq \emptyset $ is a closed convex subset of \(\mathcal{S}\) for \( b < 1 \) with \( \mathcal{F}(\daleth) \neq \emptyset \) then \( \{\varkappa_n\} \) defined by (2) is converges strongly to a fixed point of $\daleth$ if and only if \(
\liminf_{n \to \infty} d(\varkappa_n, \mathcal{F}(\daleth)) = 0,
\)
where $d(\varkappa_n, \mathcal{F}(\daleth)) = \inf \{ d(\varkappa_n, z^*) : z^* \in \mathcal{F}(\daleth) \}$.
\end{theorem}

\begin{proof}
Consider $\{\varkappa_n\}$ converges strongly to $z^* \in \mathcal{F}(\daleth)$, then \(
\lim\limits_{n \to \infty} d(\varkappa_n, z^*) = 0.
\) and $0 \leq d(\varkappa_n, \mathcal{F}(\daleth)) \leq d(\varkappa_n, z^*)$ which implies \( \lim\limits_{n \to \infty} d(\varkappa_n, \mathcal{F}(\daleth)) = 0.\) Thus $\liminf_{n \to \infty} d(\varkappa_n, \mathcal{F}(\daleth)) = 0$.

\noindent \textbf{(Conversely)} Let $\liminf_{n \to \infty} d(\varkappa_n, \mathcal{F}(\daleth)) = 0$ then $\lim\limits_{n \to \infty} d(\varkappa_n, \mathcal{F}(\daleth))$ exists (from lemma 4). Thus \[
\lim\limits_{n \to \infty} d(\varkappa_n, \mathcal{F}(\daleth)) = 0.
\] Now we need to show that $\{\varkappa_n\}$ is a Cauchy sequence in $\mathcal{A}$. Since $\lim\limits_{n \to \infty} d(\varkappa_n, \mathcal{F}(\daleth)) = 0$ for each $\epsilon > 0$ there is $n_0 \in \mathbb{N}$ for all \(n \geq n_0\) such that
\[
d(\varkappa_n, \mathcal{F}(\daleth)) < \frac{\epsilon}{2}
\]
Particularly, \[\inf \{ d(\varkappa_{n_0}, z^*) : z^* \in \mathcal{F}(\daleth) \} < \dfrac{\epsilon}{2}.\]
Thus $\exists \quad z^*_1 \in \mathcal{F}(\daleth)$ then
\[d(\varkappa_{n_0}, z^*_1) < \dfrac{\epsilon}{2}.\]

\noindent For any $m,z^* n \geq n_0$, we get
\begin{align*}
d(\varkappa_{n+m}, \varkappa_n) &\leq d(x_{n+m}, z^*_1) + d(z^*_1, \varkappa_n) \\
&\leq d(\varkappa_{n_0}, z^*_1) + d(z^*_1, \varkappa_{n_0}) \\
&\leq \frac{\epsilon}{2} + \frac{\epsilon}{2} \\
&= \epsilon.
\end{align*}
$\implies \{\varkappa_n\} \in \mathcal{A}$ is a Cauchy sequence, since $\mathcal{A} \in \mathcal{S}$ then $\mathcal{A}$ is complete. Thus, $\{\varkappa_n\}$ converge to a point in $\mathcal{A}$, i.e., \(
\lim\limits_{n \to \infty} \varkappa_n = z^*.\)
using theorem 2, $\mathcal{F}(\daleth)$ is closed and $\lim\limits_{n \to \infty} d(\varkappa_n, \mathcal{F}(\daleth)) = 0$ implies \(
\lim\limits_{n \to \infty} d(z^*, \mathcal{F}(\daleth)) = 0,
\) i.e., $z^* \in \mathcal{F}(\daleth)$.
\end{proof}

\subsection*{Declarations}

\noindent All authors contributed equally (Mohd Tariq and Mayank Sharma).

\noindent The authors declare that they have no conflict of interest.





\bibliography{sn-bibliography}

\end{document}